\documentclass[a4paper,11pt]{article}
\usepackage{graphicx}
\usepackage{geometry}
\usepackage{amsfonts}
\usepackage{amssymb}
\usepackage[english]{babel}
\usepackage{multirow}
\usepackage{multicol}
\usepackage{amsthm}
\usepackage{subfigure}
\usepackage{systeme}
\usepackage{appendix}
\usepackage{graphicx}
\usepackage{amsmath}
\usepackage{xcolor}
\usepackage{cases}
\newenvironment{myproof}[1][\proofname]{\proof[#1]\mbox{}}{\endproof}

\usepackage[latin1]{inputenc}
\usepackage{setspace}
\newtheorem{axi}{Axiom}[section]

\newtheorem{theorem}[axi]{Theorem}

\newtheorem{lema}[axi]{Lemma}
\newtheorem{corolary}[axi]{Corollary}
\newtheorem{definition}[axi]{Definition}
\newtheorem{observation}[axi]{Remark}
\newtheorem{conjecture}[axi]{Conjecture}

\begin{document}

\setcounter{page}{01}
\title{\vspace{-1in}\parbox{\linewidth}{\footnotesize\noindent
}
 \vspace{\bigskipamount} \\
 The Caffarelli-Kohn-Nirenberg Inequalities on Metric Measure Spaces
   \renewcommand{\thefootnote}{}
\footnotetext{\hspace*{2ex}\ \ \
Key words and phrases:
CKN inequalities, metric measure spaces, Riemannian manifolds.
 }
}
\date{}
\author{ Willian Isao Tokura\thanks{Supported by Capes} ,  Levi Adriano\thanks{Supported in part by FAPEG} and  Changyu Xia}
\maketitle



{\Large}

\begin{abstract}
In this paper, we prove that if a metric measure space satisfies the
volume doubling condition and the Caffarelli-Kohn-Nirenberg
inequality with same exponent $n(n\geq2)$, then it has exactly
$n$-dimensional volume growth. As application, we obtain geometric
and topological properties of Alexandrov space, Riemannian manifold
and Finsler space which support a Caffarelli-Kohn-Nirenberg
inequality.
\end{abstract}

\begin{section}{Introduction}

\hspace{0,6cm}Let $\mathbb{R}^{n}$ be the Euclidean space, denote by
$dx$ the volume element associated with canonical metric $g_{0}$ of
$\mathbb{R}^{n}$ and consider
$\mathcal{C}_{0}^{\infty}(\mathbb{R}^{n})$ the space of the smooth
functions in $\mathbb{R}^{n}$ with compact support.

Among a much more general family of inequalities, Caffarelli, Kohn
and Nirenberg proved the following result

 \begin{theorem}\label{eq:0}(\cite{Caff})Let $n\geq2$ and $p,q,r, \alpha, \beta,\gamma, \sigma, a$ be fixed real numbers satisfying:

 $$p,q\geq1, r>0, 0\leq a\leq1$$
 $$\frac{1}{p}+\frac{\alpha}{n}, \frac{1}{q}+\frac{\beta}{n},\frac{1}{r}+\frac{\gamma}{n}>0,$$
where
$$\gamma=a\sigma+(1-a)\beta,$$
$$\frac{1}{r}+\frac{\gamma}{n}=a\left(\frac{1}{p}+\frac{\alpha-1}{n}\right)+(1-a)\left(\frac{1}{q}+\frac{\beta}{n}\right)$$
and

$$0\leq \alpha-\sigma \hspace{0,2cm}\text{if}\hspace{0,2cm} a>0 \hspace{0,2cm}\text{and}$$
$$\alpha-\sigma \leq1 \hspace{0,2cm}\text{if}\hspace{0,2cm} a>0 \hspace{0,2cm}\text{and}\hspace{0,2cm}\frac{1}{p}+\frac{\alpha-1}{n}=\frac{1}{r}+\frac{\gamma}{n}.$$
There exists a positive constant $C=C(n,p,q,r,\alpha,\beta,\gamma)$
such that the following inequality holds for all $
u\in\mathcal{C}_{0}^{\infty}(\mathbb{R}^{n})$

\begin{equation}\label{eq:29}\left(\int_{\mathbb{R}^{n}}|x|^{\gamma r}|u|^{r}dx\right)^{\frac{1}{r}}\leq
C\left(\int_{\mathbb{R}^{n}}|x|^{\alpha p}|\nabla
u|^{p}dx\right)^{\frac{a}{p}}\left(\int_{\mathbb{R}^{n}}|x|^{\beta
q}|u|^{q}dx\right)^{\frac{1-a}{q}}.
\end{equation}
\end{theorem}

Let us denote by $C_{opt}(\mathbb{R}^{n})$ the best constant for
this inequality, that is,

\begin{equation}C_{opt}(\mathbb{R}^{n})^{-1}=\inf_{u\in\mathcal{C}_{0}^{\infty}(\mathbb{R}^{n})-\{0\}}\frac{{\left(\int_{\mathbb{R}^{n}}|x|^{\alpha p}|\nabla
u|^{p}dx\right)^{\frac{a}{p}}\left(\int_{\mathbb{R}^{n}}|x|^{\beta
q}|u|^{q}dx\right)^{\frac{1-a}{q}}}}{\left(\int_{\mathbb{R}^{n}}|x|^{\gamma
r}|u|^{r}dx\right)^{\frac{1}{r}}}.\nonumber
\end{equation}

Recently, in \cite{lam} the authors consider the change of
exponent in \eqref{eq:29}

\begin{equation}\alpha=-\frac{\mu}{p},\hspace{0,3cm}\beta=-\frac{\theta}{q},\hspace{0,3cm}\gamma=-\frac{s}{r}\nonumber
\end{equation}
and get the following result:

\begin{theorem}\label{eq:30}(\cite{lam}, Theorem $1.2$)Let $n\geq2$ and $p,q,\mu$ be fixed real numbers satisfying
\begin{equation}
\label{eq1} 1<p<p+\mu<n,\hspace{0,1cm} 1\leq
q<\frac{p(q-1)}{p-1}<\frac{np}{n-p},
\end{equation}
and let $r$, $\theta$, $s$ and $a$ given by
\begin{equation}
\label{eq2}
r=\frac{p(q-1)}{p-1},\hspace{0,1cm}\theta=s=\frac{n\mu}{n-p},
\hspace{0,1cm}a=\frac{n(q-p)}{(q-1)[np-q(n-p)]},
\end{equation}
Then, with $\nu=np-q(n-p),$
\begin{equation}C_{opt}(\mathbb{R}^{n})=\left(\frac{n-p}{n-p-\mu}\right)^{\frac{1}{r}+\frac{p-1}{p}-\frac{1-a}{q}-\frac{(p-1)(1-a)}{p}}\left(\frac{q-p}{p\sqrt{\pi}}\right)^{a}
\left(\frac{pq}{n(q-p)}\right)^{\frac{a}{p}}\left(\frac{\nu}{pq}\right)^{\frac{1}{r}}\left(\frac{\Gamma(q\frac{p-1}{q-p})\Gamma(\frac{n}{2}+1)}{\Gamma(\frac{p-1}{p}\frac{\delta}{q-p})\Gamma(n\frac{p-1}{p}+1)}\right)^{\frac{a}{n}},\nonumber
\end{equation}
and all minimal functions are of the form
\begin{equation}V_{0}(x)=A(1+B|x|^{\frac{n-p-\mu}{n-p}\frac{p}{p-1}})^{-\frac{p-1}{q-p}},\hspace{0,3cm} A\in\mathbb{R},
B>0.\nonumber
\end{equation}
\end{theorem}

In \cite{Adr, Adr2, doCarmo, Heb, Led, Xia1, Xia}, the authors consider the study of Riemannian
manifolds with non-negative Ricci curvature supporting some of the
particular classes of CKN. In particular, in \cite{Adr, Adr2, doCarmo, Xia1, Xia}, the authors
obtain some metric and topological rigidity results.

In the case of CKN inequality type, Xia in \cite{Xia4} considered
the case
\begin{equation}q=\frac{p(r-1)}{p-1},\hspace{0,2cm} 1<p<r,\hspace{0,2cm}
n-\theta<\left(1+\frac{\mu}{p}-\frac{\theta}{p}\right),\hspace{0,2cm}s=\frac{\mu}{p}+1+\frac{\theta(p-1)}{p}\nonumber
\end{equation}
and obtained the extremal functions, which are
$u(x)=(\lambda+|x|^{1+\frac{\mu}{p}-\frac{\theta}{p}})^{-\frac{p-1}{r-p}}$.
Furthermore,  metrical and topological theorems were obtained.

For metric measure spaces, Kristály and Ohta in \cite{Kris} and
\cite{Kris1}, study metric measure spaces supporting the
Gagliardo-Nirenberg inequality and a particular class of
Caffarelli-Kohn-Nirenberg and obtains that the metric space has
exactly the $n$-dimensional volume growth, as application they get
some rigidity theorems on Finsler geometry.


In this paper, assuming the hypotheses of Theorem \ref{eq:30} we
extend the main result of Kristály and Ohta in \cite{Kris} for class
of Caffarelli-Kohn-Nirenberg inequality and obtain some rigidity
result in Alexandrov, Riemannian and Finsler geometry. We estate our main result in the sequel.

\begin{theorem}\label{eq:45}Consider $n$, $a$, $p$, $q$, $r$, $s$, $\mu$, $\theta$ as in Theorem
\ref{eq:30}. Let (X,d,\textsf{m}) be a proper metric measure space
and assume that for some $x_{0}\in X$, $C\geq
C_{opt}(\mathbb{R}^{n})$, $C_{0}\geq 1$, the
Caffarelli-Kohn-Nirenberg \eqref{eq:31} hold on $X$ with the
following conditions

\begin{equation}\label{eq:32}\frac{\textsf{m}(B_{R}(x))}{\textsf{m}(B_{\rho}(x))}\leq
C_{0}\left(\frac{R}{\rho}\right)^{n},\hspace{0,3cm}\forall x\in
X,\hspace{0,2cm} e\hspace{0,2cm} 0<\rho<R
\end{equation}
and
\begin{equation}\label{eq:33}\liminf_{\rho\rightarrow0}\frac{\textsf{m}(B_{\rho}(x_{0}))}{\textsf{m}_{E}(\mathbb{B}_{\rho}(0))}=1,
\end{equation}
where $B_{\rho}(x):=\{y\in X: d(x,y)<\rho\}$,
$\mathbb{B}_{\rho}(0):=\{x\in \mathbb{R}^{n}: |x|<\rho\}$ and
$\textsf{m}_{E}$ is the $n$-dimensional Lebesgue measure. Then , we
have

\begin{equation}\label{eq:54}\textsf{m}(B_{\rho}(x))\geq
C_{0}^{-1}\left(\frac{C_{opt}(\mathbb{R}^{n})}{C}\right)^{\frac{n}{a}}\textsf{m}_{E}(\mathbb{B}_{\rho}(0)),\hspace{0,2cm}
\forall \rho>0, \hspace{0,2cm} x\in X.
\end{equation}
In particular
\begin{equation}
C_{0}^{-1}\left(\frac{C_{opt}(\mathbb{R}^{n})}{C}\right)^{\frac{n}{a}}w_{n}\rho^{n}\leq
\textsf{m}(B_{\rho}(x_{0}))\leq C_{0}w_{n}\rho^{n},\nonumber
\end{equation}
for all $\rho>0$, where $w_{n}$ denotes the volume of the unit ball
in $\mathbb{R}^{n}$.
\end{theorem}

\noindent In the above theorem, we consider in $(X,d)$ the Borel
measure $\textsf{m}$ such that $0<\textsf{m}(U)<\infty$ for all not
empty open set $U\subset X$ and for fixed $x_{0}\in X$ and $C>0$ we
consider the Caffarelli-Kohn-Nirenberg inequality on
$(X,d,\textsf{m})$ of the form, $\forall u\in Lip_{0}(X)$

\begin{equation}\label{eq:31}\left(\int_{X}d(x,x_{0})^{\gamma r}|u|^{r}d\textsf{m}(x)\right)^{\frac{1}{r}}\leq
C\left(\int_{X}d(x,x_{0})^{\alpha p}|D
u|^{p}d\textsf{m}(x)\right)^{\frac{a}{p}}\left(\int_{X}d(x,x_{0})^{\beta
q}|u|^{q}d\textsf{m}(x)\right)^{\frac{1-a}{q}}
\end{equation}

\noindent where $Lip_{0}(X)$ denote the space of Lipschitz functions
with compact support and

\begin{equation}|Du|(x):=\limsup_{y\rightarrow x}\frac{|u(x)-u(y)|}{d(x,y)}\nonumber
\end{equation}
is the local Lipschitz constant of $u$ at $x$.

\begin{observation}As pointed out in [\cite{Kris}, Remark 1.3(2)] if
$(X,d,\textsf{m})$ satisfies the volume doubling condition
\begin{equation}\textsf{m}(B_{2\rho}(x))\leq\Lambda\textsf{m}(B_{\rho}(x)),\hspace{0,2cm}
\text{for some} \hspace{0,2cm}\Lambda\geq1,\hspace{0,2cm} \text{and
all}\hspace{0,2cm} x\in X, \hspace{0,2cm}\rho>0\nonumber
\end{equation}
then it is easy to get that the volume condition \eqref{eq:32} is
satisfied with, e.g., $n\geq\log_{2}\Lambda$ and $C=1$. Thus
\eqref{eq:32} can be interpreted as the volume doubling condition
with the explicit exponent $n$.
\end{observation}

In the Riemannian case, we show that the constant in the
Caffarelli-Kohn-Nirenberg inequality on a complete open Riemannian
manifold should be bigger than or equal to the optimal one on the
Euclidean space of the same dimension, that is, we have the
following

\begin{theorem}\label{eq:18}Let $(M^{n},g)$ be a complete non-compact Riemannian manifold with volume element $dv$, distance function $d(x)=d(x,x_{0})$ for fixed point $x_{0}\in M$, and $n$, $a$, $\alpha$, $\beta$,
$\gamma$, $p$, $q$, $r$ constants as in Theorem \ref{eq:0}. Suppose
that there exists a constant $C\in\mathbb{R}$, such that for all
$u\in \mathcal{C}_{0}^{\infty}(M)$,

\begin{equation}\left(\int_{M}d(x)^{\gamma r}|u|^{r}dv\right)^{\frac{1}{r}}\leq C\left(\int_{M}d(x)^{\alpha p}|\nabla
u|^{p}dv\right)^{\frac{a}{p}}\left(\int_{M}d(x)^{\beta q}|u|^{q}dv\right)^{\frac{1-a}{q}}.\nonumber\\
\end{equation}
Then $C_{opt}(\mathbb{R}^{n})\leq C$.
\end{theorem}
Now, recall first the definition of asymptotically non-negative
Ricci curvature.

\begin{definition}A complete open manifold $M^{n}$ is said to have
asymptotically non-negative Ricci curvature with base point
$x_{0}\in M$ if
\begin{equation}\label{eq:21}Ric_{M,g}(x)\geq-(n-1)G(d(x)),\hspace{0,2cm}\forall
x\in M
\end{equation}
where $d(x)$ is the distance function on $M$ from $x_{0}$ and $G\in
C^{1}([0,\infty))$ is a non-negative function satisfying
\begin{equation}\int_{0}^{\infty}tG(t)dt=b_{0}<\infty.\nonumber
\end{equation}
\end{definition}

In this case, $M^{n}$ satisfies the following volume growth
property(see Corollary $2.17$ in \cite{Rig}):

\begin{equation}\label{eq:19}\frac{Vol[B_{R}(p)]}{Vol[B_{\widetilde{R}}(p)]}\leq
e^{(n-1)b_{0}}\left(\frac{R}{\widetilde{R}}\right)^{n},\hspace{0,3cm}
0<\widetilde{R}<R
\end{equation}
which implies easily that $M^{n}$ has doubling volume property at
$p$ and

\begin{equation}Vol[B_{R}(p)]\leq
e^{(n-1)b_{0}}\omega_{n}R^{n},\hspace{0,3cm} \forall R>0.\nonumber
\end{equation}
Then, as a corollary of Theorem \ref{eq:45}, we have

\begin{corolary}\label{eq:17}Let $(M^{n}, g)$ be a complete non-compact Riemannian manifold with Ricci curvature satisfying \eqref{eq:21} and suppose that for some positive constant $C>0$
\begin{equation}\label{eq:16}\left(\int_{M}d(x)^{\gamma r}|u|^{r}dv\right)^{\frac{1}{r}}\leq C\left(\int_{M}d(x)^{\alpha p}|\nabla
u|^{p}dv\right)^{\frac{a}{p}}\left(\int_{M}d(x)^{\beta
q}|u|^{q}dv\right)^{\frac{1-a}{q}}, \hspace{0,3cm}\forall
u\in\mathcal{C}_{0}^{\infty}(M)
\end{equation}
Then for all $R>0$ we have

\begin{equation}e^{-(n-1)b_{0}}\left(\frac{C_{opt}(\mathbb{R}^{n})}{C}\right)^{\frac{n}{a}}V(R)\leq Vol[B_{R}(p)]\leq e^{(n-1)b_{0}}V(R)\nonumber
\end{equation}
where $V(R)$ denotes the volume of the Euclidean ball of radius $R$
in $\mathbb{R}^{n}$.
\end{corolary}

A theorem due to Cheeger and Colding \cite{Chee} states that given
an integer $n\geq2$ there exists a constant $\delta(n)>0$ such that
any $n$-dimensional complete Riemannian manifold with non-negative
Ricci curvature and $Vol[B_{r}(x)]\geq(1-\delta(n))V(r)$ for all
$x\in M$ and all $r>0$ is diffeomorphic to $\mathbb{R}^{n}$. Then
combining this result with Corollary \ref{eq:17}, we have the
following rigidity result.

\begin{corolary}Give an integer $n\geq2$, exist $\epsilon(n)>0$ such that any complete non-compact Riemannian manifold $(M^{n},g)$
with non-negative Ricci curvature in which the inequality

\begin{equation}\left(\int_{M}d(x)^{\gamma r}|u|^{r}dv\right)^{\frac{1}{r}}\leq( C_{opt}(\mathbb{R}^{n})+\epsilon(n))\left(\int_{M}d(x)^{\alpha p}|\nabla
u|^{p}dv\right)^{\frac{a}{p}}\left(\int_{M}d(x)^{\beta q}|u|^{q}dv\right)^{\frac{1-a}{q}},\hspace{0,2cm}\forall u\in\mathcal{C}_{0}^{\infty}(M)\nonumber\\
\end{equation}
is satisfied, is diffeomorphic to $\mathbb{R}^{n}$.
\end{corolary}

From Bishop comparison theorem \cite{Chavel, Schoen}, we have that
if a complete Riemannian manifold $(M^{n},g)$ has non-negative Ricci
curvature, then for all $x\in M$, $Vol[B_{R}(x)]\leq V(R)$ and
equality hold if, and only if, $B_{R}(x)$ is isometric to Euclidean
ball $V(R)$. Thus by Corollary \ref{eq:17}, we have:

\begin{corolary}Let $(M^{n},g)$ be a complete non-compact Riemannian manifold with non-negative Ricci curvature
and suppose that the following Caffarelli-Kohn-Nirenberg inequality
hold
\begin{equation}\left(\int_{M}d(x)^{\gamma r}|u|^{r}dv\right)^{\frac{1}{r}}\leq C_{opt}(\mathbb{R}^{n})\left(\int_{M}d(x)^{\alpha p}|\nabla
u|^{p}dv\right)^{\frac{a}{p}}\left(\int_{M}d(x)^{\beta q}|u|^{q}dv\right)^{\frac{1-a}{q}},\hspace{0,3cm}u\in\mathcal{C}_{0}^{\infty}(M)\nonumber\\
\end{equation}
Then $M$ is isometric to Euclidean space $\mathbb{R}^{n}$.

\end{corolary}

It has been shown by Zhu \cite{Zhu}, that given $\delta>0$, there is
an $\epsilon(n,\delta)$ such that if a complete non-compact
Riemannian manifold $(M^{n},g)$ with sectional curvature satisfying

\begin{equation}K(x)\geq -G(d(x)),\hspace{0,2cm}
\int_{0}^{\infty}tG(t)dt\leq\epsilon\nonumber
\end{equation}
and
\begin{equation}Vol[B_{R}(p)]\geq\left(\frac{1}{2}+\delta\right)V(R),\hspace{0,2cm}\forall
R>0,\nonumber
\end{equation}
then the distance function $d=d(x_{0},.):M\rightarrow\mathbb{R}$ has
no critical points and hence $M$ is diffeomorphic to
$\mathbb{R}^{n}$. Combining this Zhu's theorem with Corollary
\ref{eq:17}, we have

\begin{corolary}Let $(M^{n},g)$ be a complete non-compact Riemannian manifold. Fix a $\delta\in(0,\frac{1}{2})$,
there exist a $b_{0}(n,\delta)>0$ such that, if the sectional
curvature of $M$ satisfies

\begin{equation}K(x)\geq -G(d(x)),\hspace{0,2cm}
\int_{0}^{\infty}tG(t)dt\leq b_{0}\nonumber
\end{equation}
and the inequality \eqref{eq:31} holds on $M$ with
$C<(\frac{1}{2}+\delta)^{-\frac{a}{n}}C_{opt}(\mathbb{R}^{n})$, then
$M$ is diffeomorphic to Euclidean space $\mathbb{R}^{n}$.
\end{corolary}

It is interesting to know under what kind of conditions a complete
$n$-dimensional metric measure space has finite topological type or
is isometric to Euclidean space $\mathbb{R}^{n}$. In the context of
Alexandrov spaces, as application of Theorem \ref{eq:45}, we prove
the following results

\begin{theorem}\label{111}Consider $n$, $a$, $p$, $q$, $r$, $s$, $\mu$, $\theta$ as in Theorem \ref{eq:30}. Let $(X,
d)$ be a complete, locally compact non-compact Alexandrov space with
non-negative curvature and measure $\lambda\mathcal{H}^{n}$, with
\linebreak$\lambda=\frac{\omega_{n}}{\mathcal{H}^{n}(B_{1}(o_{x_{0}}))}$,
where $o_{x_{0}}$ denote the vertix of the tangent cone $K_{x_{o}}M$
at $x_{o}$ and $\mathcal{H}^{n}$ is a $n$-dimensional Hausdorff
measure of $X$. Suppose that $X$ supports the CKN inequality with
\linebreak$C=C_{opt}(\mathbb{R}^{n})$ for some point $x_{0}\in X$.
Then $(X,d)$ is isometric to Euclidean space $\mathbb{R}^{n}$.
\end{theorem}

\begin{theorem}\label{112}Consider $n$, $a$, $p$, $q$, $r$, $s$, $\mu$, $\theta$ as in Theorem \ref{eq:30}. Then exist a
$\delta(n)>0$ such that any locally compact $n$-dimensional complete
Alexandrov space $(X, d)$ with curvature $\geq0$ and $n$-dimensional
Hausdorff measure $\mathcal{H}^{n}$ satisfying

\begin{equation}\liminf_{\rho\rightarrow0}\frac{\mathcal{H}^{n}(B_{\rho}(x_{0}))}{\omega_{n}\rho^{n}}=1\nonumber
\end{equation}
in which the inequality
\begin{equation}\left(\int_{X}d(x,x_{0})^{\gamma r}|u|^{r}d\mathcal{H}^{n}\right)^{\frac{1}{r}}\leq
(C_{opt}(\mathbb{R}^{n})+\delta(n))\left(\int_{X}d(x,x_{0})^{\alpha
p}|D
u|^{p}d\mathcal{H}^{n}\right)^{\frac{a}{p}}\left(\int_{X}d(x,x_{0})^{\beta
q}|u|^{q}d\mathcal{H}^{n}\right)^{\frac{1-a}{q}}\nonumber
\end{equation}
is satisfied for all $u\in Lip_{0}(X)$, has Finite topological type.
\end{theorem}

As pointed in \cite{Kris}, on Finsler manifolds with non-negative
$n$-Ricci curvature, the condition \eqref{eq:32} holds with
$C_{0}=1$. In particular, for Finsler manifolds in which a
particular class of Caffarelli-Kohn-Nirenberg inequality holds, they
get some metric rigidity theorem. Motivated by work \cite{Kris} we
obtain similar results on Finsler manifolds for a class of
Caffarelli-Kohn-Nirenberg given by Theorem \ref{eq:45}. That is, we
have

\begin{theorem}\label{eq:58}Consider $n$, $a$, $p$, $q$, $r$, $s$, $\mu$, $\theta$ as in Theorem \ref{eq:30}. Let $(M, F)$ be a complete $n$-dimensional
Finsler manifold. Fix a positive smooth measure on $M$ and assume
that the $n$-Ricci curvature $Ric_{n}$ of $(M,F,\textsf{m})$ is
non-negative, the sharp Caffarelli-Kohn-Nirenberg inequality
\eqref{eq:31} holds for some $x_{0}\in M$, and in addition

\begin{equation}\liminf_{\rho\rightarrow 0}\frac{\textsf{m}(B_{\rho}(x))}{\omega_{n}\rho^{n}}=1\nonumber
\end{equation}
for all $x\in M$. Then the Flag Curvature of $(M,F)$ is identically
zero.
\end{theorem}

\begin{theorem}\label{114}Consider $n$, $a$, $p$, $q$, $r$, $s$, $\mu$, $\theta$ as in Theorem \ref{eq:30}. Let $(M, F)$ be a complete $n$-dimensional Berwald
space with Busemann-Hausdorff measure $\textsf{m}_{BH}$ and
non-negative Ricci curvature. If for some $x_{0}\in M$ the sharp
Caffarelli-Kohn-Nirenberg inequality \eqref{eq:31} holds, then
$(M,F)$ is isometric to a  Minkowski space.
\end{theorem}

Finally, in \cite{Lak} the Author define the concept of large volume
growth on Finsler space, and conjectured the following

\begin{conjecture}A geodesically complete Berwald space $(M,F)$ of non-negative
flag curvature with large volume growth is diffeomorphic to
Euclidean space $\mathbb{R}^{n}$.
\end{conjecture}

\begin{observation}\label{eq4}Kell in \cite{Kell} gave an affirmative answer to this
conjecture, see (\cite{Kell}, Corollary $27$).
\end{observation}

As consequence of this fact, we prove that

\begin{theorem}\label{117}Consider $n$, $a$, $p$, $q$, $r$, $s$, $\mu$, $\theta$ as in Theorem \ref{eq:30}. Let $(M, F, \textsf{m}_{BH})$ be a complete $n$-dimensional Berwald
space with Busemann-Hausdorff measure $\textsf{m}_{BH}$ and
non-negative flag curvature. If for some $x_{0}\in M$ the
Caffarelli-Kohn-Nirenberg inequality \eqref{eq:31} holds on $M$ for
some constant $C\geq C_{opt}(\mathbb{R}^{n})$, then $(M,F)$ is
diffeomorphic to Euclidean space $\mathbb{R}^{n}$.
\end{theorem}

\end{section}


\begin{section}{Proof of Theorem 1.3}
Before proving Theorem $1.3$ we need the following lemma, the proof
of which is similar to the arguments used by Ledoux and Xia
(cf.\cite{Led},\cite{Xia4}-\cite{Xia}). For the sake of
completeness, we will include it.

\begin{lema}\label{eq:46}Let $(X,d,\textsf{m})$ be a proper metric measure space with measure satisfying the conditions \eqref{eq:32} and
\eqref{eq:33} in Theorem \ref{eq:45} for some point $x_{0}\in X$.
Let $d(x)=d(x,x_{0})$, and suppose that the inequality \eqref{eq:31}
holds on $X$ for some constant $C>C_{opt}(\mathbb{R}^{n})$. Then,
for all $\lambda>0$
\begin{equation}F(\lambda)\geq\left(\frac{C_{opt}(\mathbb{R}^{n})}{C}\right)^{\frac{n}{a}}G(\lambda)\nonumber
\end{equation}
where
\begin{equation}\label{eq:36}
F(\lambda)=\frac{q-p}{r(p-1)-(q-p)}\int_{X}\frac{d(x)^{ \gamma
r}}{(\lambda+d(x)^{\frac{n-p-\mu}{n-p}\frac{p}{p-1}})^{\frac{q(p-1)}{q-p}}}d\textsf{m}(x)
\end{equation}
and

\begin{equation}
G(\lambda)=\frac{q-p}{r(p-1)-(q-p)}\int_{\mathbb{R}^{n}}\frac{|x|^{
\gamma
r}}{(\lambda+|x|^{\frac{n-p-\mu}{n-p}\frac{p}{p-1}})^{\frac{q(p-1)}{q-p}}}d\textsf{m}_{E}(x).\nonumber
\end{equation}
\end{lema}

\begin{proof}Firstly observe that $F$ is well defined and of class $C^{1}$. Indeed, by Fubini's theorem
(See \cite{Schoen})

\begin{equation}
F(\lambda)=\frac{q-p}{r(p-1)-(q-p)}\int_{0}^{\infty}\textsf{m}\Bigg{\{}x:\frac{d(x)^{
\gamma
r}}{(\lambda+d(x)^{\frac{n-p-\mu}{n-p}\frac{p}{p-1}})^{\frac{q(p-1)}{q-p}}}>s\Bigg{\}}ds.
\end{equation}
By the process of change of variable of the form

\begin{equation}
s=\frac{h^{ \gamma
r}}{(\lambda+h^{\frac{n-p-\mu}{n-p}\frac{p}{p-1}})^{\frac{q(p-1)}{q-p}}}\nonumber
\end{equation}
we get

\begin{eqnarray}\label{eq:37}
F(\lambda)&=&\frac{q-p}{r(p-1)-(q-p)}\int_{0}^{\infty}\textsf{m}\Big{\{}x:d(x)<h\Big{\}}h^{\gamma
r-1}\frac{\Big{[}-\gamma r\lambda+(\frac{pq}{q-p}\frac{n-p-\mu}{n-p}-\gamma r)h^{\frac{n-p-\mu}{n-p}\frac{p}{p-1}}\Big{]}}{(\lambda+h^{\frac{n-p-\mu}{n-p}\frac{p}{p-1}})^{\frac{q(p-1)}{q-p}+1}}dh\nonumber\\
&=&\frac{q-p}{r(p-1)-(q-p)}\int_{0}^{\infty}\textsf{m}(B_{h}(x_{0}))h^{\gamma
r-1}\frac{\Big{[}-\gamma
r\lambda+(\frac{pq}{q-p}\frac{n-p-\mu}{n-p}-\gamma
r)h^{\frac{n-p-\mu}{n-p}\frac{p}{p-1}}\Big{]}}{(\lambda+h^{\frac{n-p-\mu}{n-p}\frac{p}{p-1}})^{\frac{q(p-1)}{q-p}+1}}dh.
\end{eqnarray}
The hypothesis \eqref{eq:32} and \eqref{eq:33} implies that
$\textsf{m}(B_{h}(x_{0}))\leq Ah^{n}$, $\forall h>0$, for some
positive constant $A\in \mathbb{R}$. Thus

\begin{equation}F(\lambda)\leq\frac{(q-p)A}{r(p-1)-(q-p)}\int_{0}^{\infty}h^{n+\gamma
r-1}\frac{\Big{[}-\gamma
r\lambda+(\frac{pq}{q-p}\frac{n-p-\mu}{n-p}-\gamma
r)h^{\frac{n-p-\mu}{n-p}\frac{p}{p-1}}\Big{]}}{(\lambda+h^{\frac{n-p-\mu}{n-p}\frac{p}{p-1}})^{\frac{q(p-1)}{q-p}+1}}dh.\nonumber
\end{equation}

From (\ref{eq1}) and (\ref{eq2}) we have that

$$
n+\gamma r-1>-1
$$
and
$$
n+\gamma r-1-\left(\frac{n-p-\mu}{n-p}\right)\frac{pq}{p-1}\frac{p-1}{q-p}<-1.
$$
Therefore, $0\leq F(\lambda)<\infty$, $\forall \lambda>0$ and $F$ is
differentiable. Also, we have

\begin{equation}\label{eq:35}F'(\lambda)=-\int_{X}\frac{d(x)^{\gamma
r}}{(\lambda+d(x)^{\frac{n-p-\mu}{n-p}\frac{p}{p-1}})^{\frac{r(p-1)}{q-p}}}d\textsf{m}(x).
\end{equation}

For every $\lambda >0$ consider the sequence of functions
$u_{\lambda,k}:X\rightarrow \mathbb{R}$, $k\in\mathbb{N}$ defined by

\begin{equation}u_{\lambda,k}(x):=\max\{0,\min\{0,k-d(x)\}+1\}\left(\lambda+\max\left\{d(x),\frac{1}{k}\right\}^{\frac{n-p-\mu}{n-p}\frac{p}{p-1}}\right)^{-\frac{(p-1)}{q-p}}.\nonumber
\end{equation}
Note that since $(X,d)$ is proper, the set
$supp(u_{\lambda,k})=\{x\in X:d(x)\leq k+1\}$ is compact. Therefore,
$u_{\lambda,k}\in Lip_{0}(X)$ for all $\lambda>0$ and
$k\in\mathbb{N}$. Consequently, consider the limit

\begin{equation}u_{\lambda}(x):=\lim_{k\rightarrow\infty}u_{\lambda,k}(x)=\left(\lambda+d(x)^{\frac{n-p-\mu}{n-p}\frac{p}{p-1}}\right)^{-\frac{(p-1)}{q-p}}.\nonumber
\end{equation}
Since the functions $u_{\lambda,k}$ satisfy the inequality
\eqref{eq:31}, we have by an approximation procedure that we can
apply $u_{\lambda}(x)$ for every $\lambda$ to \eqref{eq:31} to get

\begin{align}
\Bigg{(}\int_{X}&{\frac{d(x)^{\gamma r}}{(\lambda+d(x)^{\frac{n-p-\mu}{n-p}\frac{p}{p-1}})^{\frac{r(p-1)}{q-p}}}d\textsf{m}(x)}\Bigg{)}^{\frac{1}{r}} \nonumber \\
&\leq
C\left(\frac{p(n-p-\mu)}{(n-p)(q-p)}\right)^{a}\left(\int_{X}{d(x)^{\alpha
p}d(x)^{\frac{p(n-p(\mu+1))}{(n-p)(p-1)}}(\lambda+d(x)^{\frac{n-p-\mu}{n-p}\frac{p}{p-1}})^{\frac{p(1-q)}{q-p}}d\textsf{m}(x)}\right)^{\frac{a}{p}}\nonumber\times\\
&\times\left(\int_{X}\frac{d(x)^{\beta
q}}{(\lambda+d(x)^{\frac{n-p-\mu}{n-p}\frac{p}{p-1}})^{\frac{q(p-1)}{q-p}}}d\textsf{m}(x)\right)^{\frac{1-a}{q}}\nonumber
\end{align}
which, combining with \eqref{eq:35} gives

\begin{equation}(-F'(\lambda))^{\frac{1}{r}}\leq
C\Bigg{(}\frac{p(n-p-\mu)}{(n-p)(q-p)}\Bigg{)}^{a}\Bigg{(}\frac{r(p-1)-(q-p)}{q-p}F(\lambda)+\lambda
F'(\lambda)\Bigg{)}^{\frac{a}{p}}\Bigg{(}F(\lambda)\frac{r(p-1)-(q-p)}{q-p}\Bigg{)}^{\frac{1-a}{q}}.\nonumber
\end{equation}
Hence, $F$ satisfies the following differential inequality

\begin{equation}\label{eq:42}(-F'(\lambda))^{\frac{p}{ar}}\leq
\Gamma\Bigg{(}\frac{r(p-1)-(q-p)}{q-p}F(\lambda)+\lambda
F'(\lambda)\Bigg{)}F(\lambda)^{\frac{p(1-a)}{aq}}
\end{equation}
where

\begin{equation}
\Gamma:=C^{\frac{p}{a}}\Bigg{(}\frac{p(n-p-\mu)}{(n-p)(q-p)}\Bigg{)}^{p}\Bigg{(}\frac{r(p-1)-(q-p)}{q-p}\Bigg{)}^{\frac{p(1-a)}{aq}}.\nonumber
\end{equation}
By definition, we can easily get that

\begin{equation}\label{eq:38}G(\lambda)=\frac{\omega_{n}(q-p)}{r(p-1)-(q-p)}\int_{0}^{\infty}t^{n+\gamma
r-1}\frac{\Big{[}-\gamma
r\lambda+(\frac{pq}{q-p}\frac{n-p-\mu}{n-p}-\gamma
r)t^{\frac{n-p-\mu}{n-p}\frac{p}{p-1}}\Big{]}}{(\lambda+t^{\frac{n-p-\mu}{n-p}\frac{p}{p-1}})^{\frac{q(p-1)}{q-p}+1}}dt.
\end{equation}

Now, note that for each $\lambda>0$ the function
$z_{\lambda}:\mathbb{R}^{n}\rightarrow\mathbb{R}$ defined by
$z_{\lambda}(x)=(\lambda+|x|^{\frac{n-p-\mu}{n-p}\frac{p}{p-1}})^{-\frac{(p-1)}{q-p}}$
is an extremal function of Caffarelli-Kohn-Nirenberg on
$\mathbb{R}^{n}$, that is,

\begin{equation}\left(\int_{\mathbb{R}^{n}}|x|^{\gamma
r}|z_{\lambda}|^{r}d\textsf{m}_{E}(x)\right)^{\frac{1}{r}}=
C_{opt}(\mathbb{R}^{n})\left(\int_{\mathbb{R}^{n}}|x|^{\alpha
p}|\nabla
z_{\lambda}|^{p}d\textsf{m}_{E}(x)\right)^{\frac{a}{p}}\left(\int_{\mathbb{R}^{n}}|x|^{\beta
q}|z_{\lambda}|^{q}d\textsf{m}_{E}(x)\right)^{\frac{1-a}{q}}\nonumber
\end{equation}
and by the previously arguments, the above equality can be expressed
by

\begin{equation}\label{eq:39}(-G'(\lambda))^{\frac{p}{ar}}=
\widetilde{\Gamma}\Bigg{(}\frac{r(p-1)-(q-p)}{q-p}G(\lambda)+\lambda
G'(\lambda)\Bigg{)}G(\lambda)^{\frac{p(1-a)}{aq}}
\end{equation}
where

\begin{equation}
\widetilde{\Gamma}:=C_{opt}(\mathbb{R}^{n})^{\frac{p}{a}}\Bigg{(}\frac{p(n-p-\mu)}{(n-p)(q-p)}\Bigg{)}^{p}\Bigg{(}\frac{r(p-1)-(q-p)}{q-p}\Bigg{)}^{\frac{p(1-a)}{aq}}.\nonumber
\end{equation}
Substituting

\begin{equation}\label{eq:52}G(\lambda)=\lambda^{\frac{(q-p)(p-1)n-pq(p-1)}{p(q-p)}}G(1)
\end{equation}
into \eqref{eq:39}, we have


\begin{align}\label{eq:40}
\Bigg{(}&-\frac{(q-p)(p-1)n-pq(p-1)}{p(q-p)}\Bigg{)}^{\frac{p}{ar}}\\
\nonumber
&=C_{opt}(\mathbb{R}^{n})^{\frac{p}{a}}\Bigg{[}\left(\frac{p(n-p-\mu)}{(n-p)(q-p)}\right)^{p}\left(\frac{q(p-1)}{q-p}\right)^{\frac{p(1-a)}{aq}}\left(\frac{n(p-1)}{p}\right)G(1)^{\frac{p}{n}}\Bigg{]}.
\end{align}
Consider the constant $A\in\mathbb{R}$ given by
\begin{align}\label{eq:41}
\Bigg{(}&-\frac{(q-p)(p-1)n-pq(p-1)}{p(q-p)}\Bigg{)}^{\frac{p}{ar}}\\
\nonumber
&=C^{\frac{p}{a}}\Bigg{[}\left(\frac{p(n-p-\mu)}{(n-p)(q-p)}\right)^{p}\left(\frac{q(p-1)}{q-p}\right)^{\frac{p(1-a)}{aq}}\left(\frac{n(p-1)}{p}\right)A^{\frac{p}{n}}\Bigg{]}.
\end{align}
By a direct calculation you can easily verify that the function

\begin{equation}H_{0}(\lambda)=A\lambda^{\frac{(q-p)(p-1)n-pq(p-1)}{p(q-p)}},\hspace{0,3cm}\lambda\in(0,\infty)\nonumber
\end{equation}
satisfies the differential equation

\begin{equation}\label{eq:44}(-H_{0}'(\lambda))^{\frac{p}{ar}}=
\Gamma\Bigg{(}\frac{r(p-1)-(q-p)}{q-p}H_{0}(\lambda)+\lambda
H_{0}'(\lambda)\Bigg{)}H_{0}(\lambda)^{\frac{p(1-a)}{aq}}
\end{equation}
where
\begin{equation}
\Gamma=C^{\frac{p}{a}}\Bigg{(}\frac{p(n-p-\mu)}{(n-p)(q-p)}\Bigg{)}^{p}\Bigg{(}\frac{r(p-1)-(q-p)}{q-p}\Bigg{)}^{\frac{p(1-a)}{aq}}.\nonumber
\end{equation}
It follows from \eqref{eq:40} and \eqref{eq:41} that

\begin{equation}A=\left(\frac{C_{opt}(\mathbb{R}^{n})}{C}\right)^{\frac{n}{a}}G(1)\nonumber
\end{equation}
and so,

\begin{eqnarray}\label{eq:101}
H_{0}(\lambda)&=&A\lambda^{\frac{(q-p)(p-1)n-pq(p-1)}{p(q-p)}}\nonumber\\
&=&\left(\frac{C_{opt}(\mathbb{R}^{n})}{C}\right)^{\frac{n}{a}}G(1)\lambda^{\frac{(q-p)(p-1)n-pq(p-1)}{p(q-p)}}\nonumber\\
&=&\left(\frac{C_{opt}(\mathbb{R}^{n})}{C}\right)^{\frac{n}{a}}G(\lambda).
\end{eqnarray}
Now, we claim that if $F(\lambda_{0})<H_{0}(\lambda_{0})$ for some
$\lambda_{0}>0$ then $F(\lambda)<H_{0}(\lambda)$, $\forall
\lambda\in (0,\lambda_{0}]$. Indeed, suppose that there exists some
$\lambda_{1}\in(0,\lambda_{0})$ such that $F(\lambda_{1})\geq
H_{0}(\lambda_{1})$ and set

\begin{equation}\lambda_{2}:=\sup\{\lambda<\lambda_{0}; F(\lambda)\geq
H_{0}(\lambda)\}.\nonumber
\end{equation}
Then $F(\lambda)\leq H_{0}(\lambda)$ for all
$\lambda\in[\lambda_{2},\lambda_{0}]$, and so, we have from
\eqref{eq:42} that

\begin{eqnarray}\label{eq:43}(-F'(\lambda))^{\frac{p}{ar}}&\leq&
\Gamma\Bigg{(}\frac{r(p-1)-(q-p)}{q-p}F(\lambda)+\lambda
F'(\lambda)\Bigg{)}F(\lambda)^{\frac{p(1-a)}{aq}}\nonumber\\
&\leq&\Gamma\Bigg{(}\frac{r(p-1)-(q-p)}{q-p}H_{0}(\lambda)+\lambda
F'(\lambda)\Bigg{)}H_{0}(\lambda)^{\frac{p(1-a)}{aq}}.
\end{eqnarray}
For each $\lambda>0$, consider the function
$\varphi_{\lambda}:[0,\infty)\rightarrow\mathbb{R}$ defined by

\begin{equation}\varphi_{\lambda}(t)=t^{\frac{p}{ar}}+t\lambda\Gamma
H_{0}(\lambda)^{\frac{p(1-a)}{aq}}.\nonumber
\end{equation}
Thus, by \eqref{eq:44} and \eqref{eq:43}, we have

\begin{eqnarray}\varphi_{\lambda}(-F'(\lambda))&=&(-F'(\lambda))^{\frac{p}{ar}}-\Gamma\lambda
F'(\lambda)H_{0}(\lambda)^{\frac{p(1-a)}{aq}}\nonumber\\
&\leq&\Gamma\left(\frac{r(p-1)-(q-p)}{q-p}\right)H_{0}(\lambda)H_{0}(\lambda)^{\frac{p(1-a)}{aq}}\nonumber\\
&=&(-H_{0}'(\lambda))^{\frac{p}{ar}}-\Gamma\lambda
H_{0}'(\lambda)H_{0}(\lambda)^{\frac{p(1-a)}{aq}}\nonumber\\
&=&\varphi_{\lambda}(-H_{0}'(\lambda)).\nonumber
\end{eqnarray}
For each fixed $\lambda>0$ we can easily notice that
$\varphi_{\lambda}$ is a non-decreasing function, so we conclude by
the above inequality that

\begin{equation}-F'(\lambda)\leq-H_{0}'(\lambda), \hspace{0,2cm}
\forall \lambda\in[\lambda_{2},\lambda_{0}]
\end{equation}
consequently

\begin{equation}0\leq(F-H_{0})(\lambda_{2})\leq(F-H_{0})(\lambda_{0})<0\nonumber
\end{equation}
which is a contradiction.

By the condition \eqref{eq:33}, we know that given $\epsilon>0$,
there exist $\delta>0$ such that

\begin{equation}h\leq\delta\hspace{0,2cm}\Rightarrow\hspace{0,2cm}
(1-\epsilon)\textsf{m}_{E}(\mathbb{B}_{h}(0))\leq
\textsf{m}(B_{h}(x_{0}))\nonumber
\end{equation}
It then folows that

\begin{eqnarray}
F(\lambda)&=&\frac{q-p}{r(p-1)-(q-p)}\int_{0}^{\infty}\textsf{m}(B_{h}(x_{0}))h^{\gamma
r-1}\frac{\Big{[}-\gamma r\lambda+(\frac{pq}{q-p}\frac{n-p-\mu}{n-p}-\gamma r)h^{\frac{n-p-\mu}{n-p}\frac{p}{p-1}}\Big{]}}{(\lambda+h^{\frac{n-p-\mu}{n-p}\frac{p}{p-1}})^{\frac{q(p-1)}{q-p}+1}}dh\nonumber\\
&\geq&\frac{q-p}{r(p-1)-(q-p)}\int_{0}^{\delta}\textsf{m}(B_{h}(x_{0}))h^{\gamma
r-1}\frac{\Big{[}-\gamma r\lambda+(\frac{pq}{q-p}\frac{n-p-\mu}{n-p}-\gamma r)h^{\frac{n-p-\mu}{n-p}\frac{p}{p-1}}\Big{]}}{(\lambda+h^{\frac{n-p-\mu}{n-p}\frac{p}{p-1}})^{\frac{q(p-1)}{q-p}+1}}dh\nonumber\\
&\geq&\frac{(q-p)(1-\epsilon)}{r(p-1)-(q-p)}\int_{0}^{\delta}\textsf{m}_{E}(\mathbb{B}_{h}(0))h^{\gamma
r-1}\frac{\Big{[}-\gamma r\lambda+(\frac{pq}{q-p}\frac{n-p-\mu}{n-p}-\gamma r)h^{\frac{n-p-\mu}{n-p}\frac{p}{p-1}}\Big{]}}{(\lambda+h^{\frac{n-p-\mu}{n-p}\frac{p}{p-1}})^{\frac{q(p-1)}{q-p}+1}}dh\nonumber\\
&\geq&\Theta\int_{0}^{\Delta}\textsf{m}_{E}(\mathbb{B}_{s}(0))s^{\gamma
r-1}\frac{\Big{[}-\gamma r+(\frac{pq}{q-p}\frac{n-p-\mu}{n-p}-\gamma
r)s^{\frac{n-p-\mu}{n-p}\frac{p}{p-1}}\Big{]}}{(1+s^{\frac{n-p-\mu}{n-p}\frac{p}{p-1}})^{\frac{q(p-1)}{q-p}+1}}ds\nonumber
\end{eqnarray}
where
\begin{equation}\Theta=\frac{(q-p)(1-\epsilon)}{r(p-1)-(q-p)}\lambda^{\frac{(q-p)(p-1)n-pq(p-1)}{p(q-p)}}\hspace{1cm} \mbox{and} \hspace{1cm} \Delta=\frac{\delta}{\lambda^{\frac{(n-p)(p-1)}{p(n-p-\mu)}}}.\nonumber
\end{equation}
On the other hand, from \eqref{eq:38}, we have

\begin{eqnarray}
G(\lambda)&=&\frac{q-p}{r(p-1)-(q-p)}\int_{\mathbb{R}^{n}}\frac{|x|^{
\gamma
r}}{(\lambda+|x|^{\frac{n-p-\mu}{n-p}\frac{p}{p-1}})^{\frac{q(p-1)}{q-p}}}d\textsf{m}_{E}(x)\nonumber\\
&=&\frac{\Theta}{1-\epsilon}\int_{0}^{\infty}\textsf{m}_{E}(\mathbb{B}_{s}(0))s^{\gamma
r-1}\frac{\Big{[}-\gamma r+(\frac{pq}{q-p}\frac{n-p-\mu}{n-p}-\gamma
r)s^{\frac{n-p-\mu}{n-p}\frac{p}{p-1}}\Big{]}}{(1+s^{\frac{n-p-\mu}{n-p}\frac{p}{p-1}})^{\frac{q(p-1)}{q-p}+1}}ds.\nonumber
\end{eqnarray}
Thus

\begin{equation}\frac{F(\lambda)}{G(\lambda)}\geq(1-\epsilon)\frac{\int_{0}^{\Delta}\textsf{m}_{E}(\mathbb{B}_{s}(0))s^{\gamma r-1}\frac{\Big{[}-\gamma
r+(\frac{pq}{q-p}\frac{n-p-\mu}{n-p}-\gamma
r)s^{\frac{n-p-\mu}{n-p}\frac{p}{p-1}}\Big{]}}{(1+s^{\frac{n-p-\mu}{n-p}\frac{p}{p-1}})^{\frac{q(p-1)}{q-p}+1}}ds}{\int_{0}^{\infty}\textsf{m}_{E}(\mathbb{B}_{s}(0))s^{\gamma
r-1}\frac{\Big{[}-\gamma r+(\frac{pq}{q-p}\frac{n-p-\mu}{n-p}-\gamma
r)s^{\frac{n-p-\mu}{n-p}\frac{p}{p-1}}\Big{]}}{(1+s^{\frac{n-p-\mu}{n-p}\frac{p}{p-1}})^{\frac{q(p-1)}{q-p}+1}}ds}.
\end{equation}
Hence

\begin{equation}\liminf_{\lambda\rightarrow0}\frac{F(\lambda)}{G(\lambda)}\geq
1-\epsilon.\nonumber
\end{equation}
Letting $\epsilon\rightarrow 0$, we get
\begin{equation}\label{eq:102}\liminf_{\lambda\rightarrow0}\frac{F(\lambda)}{G(\lambda)}\geq 1.
\end{equation}

Now, since $C_{opt}(\mathbb{R}^{n})<C$, we have from \eqref{eq:101}
and \eqref{eq:102} that

\begin{eqnarray}\liminf_{\lambda\rightarrow0}\frac{F(\lambda)}{H_{0}(\lambda)}&=&\liminf_{\lambda\rightarrow0}\frac{F(\lambda)}{G(\lambda)}\left(\frac{C}{C_{opt}(\mathbb{R}^{n})}\right)^{\frac{n}{a}}\nonumber\\
&\geq&\left(\frac{C}{C_{opt}(\mathbb{R}^{n})}\right)^{\frac{n}{a}}\nonumber\\
&>&1.\nonumber
\end{eqnarray}
The above claim implies that
\begin{equation}F(\lambda)\geq H_{0}(\lambda), \hspace{0,3cm}\forall
\lambda>0\nonumber
\end{equation}
that is,

\begin{equation}F(\lambda)\geq \left(\frac{C_{opt}(\mathbb{R}^{n})}{C}\right)^{\frac{n}{a}}G(\lambda), \hspace{0,3cm}\forall
\lambda>0.\nonumber
\end{equation}
\end{proof}


\begin{myproof}[\textbf{Proof of Theorem 1.3}]Let us separate the proof into two
cases.

Case 1: $C>C_{opt}(\mathbb{R}^{n})$. In order to simplify the
calculations we will consider:

\begin{equation}\psi(h)=h^{\gamma
r-1}\frac{\Big{[}-\gamma
r\lambda+(\frac{pq}{q-p}\frac{n-p-\mu}{n-p}-\gamma
r)h^{\frac{n-p-\mu}{n-p}\frac{p}{p-1}}\Big{]}}{(\lambda+h^{\frac{n-p-\mu}{n-p}\frac{p}{p-1}})^{\frac{q(p-1)}{q-p}+1}}.
\end{equation}
From Lemma \ref{eq:46}, we know that

\begin{equation}\label{eq:48}\int_{0}^{\infty}[\textsf{m}(B_{h}(x_{0}))-d_{1}\textsf{m}_{E}(\mathbb{B}_{h}(0))]\psi(h)dh\geq0
\end{equation}
where

\begin{equation}d_{1}:=\left(\frac{C_{opt}(\mathbb{R}^{n})}{C}\right)^{\frac{n}{a}}.\nonumber
\end{equation}
From \eqref{eq:32}, for fixed $h>0$ , we have

\begin{equation}\frac{\textsf{m}(B_{R}(x_{0}))}{\textsf{m}_{E}(\mathbb{B}_{R}(0))}\leq
C_{0}\frac{\textsf{m}(B_{h}(x_{0}))}{\textsf{m}_{E}(\mathbb{B}_{h}(0))},\hspace{0,2cm}
\forall R>h\geq0.\nonumber
\end{equation}
Thus, consider
\begin{equation}d_{0}:=\limsup_{R\rightarrow\infty}\frac{\textsf{m}(B_{R}(x_{0}))}{\textsf{m}_{E}(\mathbb{B}_{R}(0))}.\nonumber
\end{equation}

Note that to prove \eqref{eq:54} in the case where
$C_{opt}(\mathbb{R}^{n})<C$, it is sufficient to prove that
$d_{1}\leq d_{0}$. We argue by contradiction, suppose that
$d_{0}<d_{1}$, then by definition of $d_{0}$, there exist
$\epsilon_{0}>0$ such that for some $h_{0}>0$,

\begin{equation}\label{eq:47}\frac{\textsf{m}(B_{h}(x_{0}))}{\textsf{m}_{E}(\mathbb{B}_{h}(0))}\leq
d_{1}-\epsilon_{0},\hspace{0,2cm} \forall h\geq h_{0}.
\end{equation}
It follows from \eqref{eq:32} and \eqref{eq:33} that

\begin{equation}\label{eq:49}\textsf{m}(B_{h}(x_{0}))\leq C_{0}\textsf{m}_{E}(\mathbb{B}_{h}(0)).
\end{equation}
Hence, substituting \eqref{eq:47} into \eqref{eq:48} and considering
inequality \eqref{eq:49}, we have

\begin{eqnarray}\label{eq:50}0&\leq&\int_{0}^{\infty}[\textsf{m}(B_{h}(x_{0}))-d_{1}\textsf{m}_{E}(\mathbb{B}_{h}(0))]\psi(h)dh\nonumber\\
&\leq&\int_{0}^{h_{0}}\textsf{m}(B_{h}(x_{0}))\psi(h)dh+(d_{1}-\epsilon_{0})\int_{h_{0}}^{\infty}\textsf{m}_{E}(\mathbb{B}_{h}(0))\psi(h)dh-d_{1}\int_{0}^{\infty}\textsf{m}_{E}(\mathbb{B}_{h}(0))\psi(h)dh\nonumber\\
&\leq&C_{0}\int_{0}^{h_{0}}\textsf{m}(\mathbb{B}_{h}(0))\psi(h)dh-d_{1}\int_{0}^{h_{0}}\textsf{m}_{E}(\mathbb{B}_{h}(0))\psi(h)dh-\epsilon_{0}\int_{h_{0}}^{\infty}\textsf{m}_{E}(\mathbb{B}_{h}(0))\psi(h)dh\nonumber\\
&=&(C_{0}-d_{1}+\epsilon_{0})\int_{0}^{h_{0}}\textsf{m}(\mathbb{B}_{h}(0))\psi(h)dh-\epsilon_{0}\int_{0}^{\infty}\textsf{m}(\mathbb{B}_{h}(0))\psi(h)dh\nonumber\\
&=&(C_{0}-d_{1}+\epsilon_{0})\int_{0}^{h_{0}}\textsf{m}(\mathbb{B}_{h}(0))\psi(h)dh-\epsilon_{0}\left(\frac{r(p-1)-(q-p)}{q-p}\right)G(\lambda).
\end{eqnarray}
Since $\lambda\leq(\lambda+h^{\frac{n-p-\mu}{n-p}\frac{p}{p-1}})$,
we have that
$\frac{1}{(\lambda+h^{\frac{n-p-\mu}{n-p}\frac{p}{p-1}})}\leq\frac{1}{\lambda}$,
and then

\begin{eqnarray}\label{eq:51}\int_{0}^{h_{0}}h^{n}\psi(h)dh&=&\int_{0}^{h_{0}}h^{n+\gamma
r-1}\frac{\Big{[}-\gamma r\lambda+(\frac{pq}{q-p}\frac{n-p-\mu}{n-p}-\gamma r)h^{\frac{n-p-\mu}{n-p}\frac{p}{p-1}}\Big{]}}{(\lambda+h^{\frac{n-p-\mu}{n-p}\frac{p}{p-1}})^{\frac{q(p-1)}{q-p}+1}}dh\nonumber\\
&\leq&\int_{0}^{h_{0}}h^{n+\gamma r-1}\frac{\Big{[}-\gamma
r\lambda+(\frac{pq}{q-p}\frac{n-p-\mu}{n-p}-\gamma r)h^{\frac{n-p-\mu}{n-p}\frac{p}{p-1}}\Big{]}}{\lambda^{\frac{q(p-1)}{q-p}+1}}dh\nonumber\\
&=&\lambda^{-\frac{q(p-1)}{q-p}-1}\int_{0}^{h_{0}}h^{n+\gamma
r-1}\Big{[}-\gamma r\lambda+(\frac{pq}{q-p}\frac{n-p-\mu}{n-p}-\gamma r)h^{\frac{n-p-\mu}{n-p}\frac{p}{p-1}}\Big{]}dh\nonumber\\
&=&\lambda^{-\frac{q(p-1)}{q-p}-1}\Bigg{[}-\frac{\gamma r\lambda
h_{0}^{n+\gamma r}}{n+\gamma
r}+\frac{(\frac{pq}{q-p}\frac{n-p-\mu}{n-p}-\gamma r)h_{0}^{n+\gamma
r+\frac{n-p-\mu}{n-p}\frac{p}{p-1}}}{n+\gamma
r+\frac{n-p-\mu}{n-p}\frac{p}{p-1}}\Bigg{]}.
\end{eqnarray}
Substituting \eqref{eq:51} and \eqref{eq:52} into \eqref{eq:50}, we
have

\begin{equation}\label{eq:53}\frac{\epsilon_{0}\left(\frac{r(p-1)-(q-p)}{q-p}\right)G(1)}{\omega_{n}\left(C_{0}-d_{1}+\epsilon_{0}\right)}\leq
\lambda^{\eta}\Bigg{[}-\frac{\gamma r\lambda h_{0}^{n+\gamma
r}}{n+\gamma r}+\frac{(\frac{pq}{q-p}\frac{n-p-\mu}{n-p}-\gamma
r)h_{0}^{n+\gamma r+\frac{n-p-\mu}{n-p}\frac{p}{p-1}}}{n+\gamma
r+\frac{n-p-\mu}{n-p}\frac{p}{p-1}}\Bigg{]}
\end{equation}
where
\begin{equation}\eta=-\frac{q(p-1)}{q-p}-1-\frac{(q-p)(p-1)n-pq(p-1)}{p(q-p)}.\nonumber
\end{equation}
But,

\begin{equation}\eta <0\hspace{0,2cm} \text{and}\hspace{0,2cm}\eta+1=-\frac{n(p-1)}{p}<0.\nonumber
\end{equation}
Then, letting $\lambda\rightarrow\infty$ one obtains a contradiction
by \eqref{eq:53}. This complete the proof of Theorem $1.3$ in the
case $C>C_{opt}(\mathbb{R}^{n})$.

Case 2: $C=C_{opt}(\mathbb{R}^{n})$. In this case we have for any
fixed $\delta>0$ that

\begin{equation}\left(\int_{X}d(x,x_{0})^{\gamma
r}|u|^{r}d\textsf{m}(x)\right)^{\frac{1}{r}}\leq(
C_{opt}(\mathbb{R}^{n})+\delta)\left(\int_{X}d(x,x_{0})^{\alpha
p}|D
u|^{p}d\textsf{m}(x)\right)^{\frac{a}{p}}\left(\int_{X}d(x,x_{0})^{\beta
q}|u|^{q}d\textsf{m}(x)\right)^{\frac{1-a}{q}}.\nonumber\\
\end{equation}
Thus, we have from \textit{Case} $1$ that

\begin{equation}\textsf{m}(B_{\rho}(x))\geq
C_{0}^{-1}\left(\frac{C_{opt}(\mathbb{R}^{n})}{C_{opt}(\mathbb{R}^{n})+\delta}\right)^{\frac{n}{a}}\textsf{m}_{E}(\mathbb{B}_{\rho}(0)),\hspace{0,2cm}
\forall \rho>0,\hspace{0,2cm} and\hspace{0,2cm} x\in X.\nonumber
\end{equation}
Letting $\delta\rightarrow0$, one obtains that

\begin{equation}\textsf{m}(B_{\rho}(x))\geq
C_{0}^{-1}\textsf{m}_{E}(\mathbb{B}_{\rho}(0)),\hspace{0,2cm}
\forall \rho>0,\hspace{0,2cm} and\hspace{0,2cm} x\in X.\nonumber
\end{equation}
This completes the proof of theorem $1.3$.
\end{myproof}

\end{section}

\begin{section}{Proof of Theorem \ref{eq:18}}

\begin{proof}We argue by contradiction, suppose that
$C<C_{opt}(\mathbb{R}^{n})$ and
\begin{equation}\left(\int_{M}d(x)^{\gamma r}|u|^{r}dv\right)^{\frac{1}{r}}\leq C\left(\int_{M}d(x)^{\alpha p}|\nabla
u|^{p}dv\right)^{\frac{a}{p}}\left(\int_{M}d(x)^{\beta
q}|u|^{q}dv\right)^{\frac{1-a}{q}} , \hspace{0,3cm}\forall u\in
\mathcal{C}_{0}^{\infty}(M)\label{eq:01}
\end{equation}

Given $\epsilon>0$ there exist a chart $(\Omega,\phi)$ of $M$ at
$x_{0}$ and a $\delta>0$ such that $\phi(\Omega)=B_{\delta}(0)$, the
Euclidean ball of radius $\delta$ centered at the origin in
$\mathbb{R}^{n}$, and that the components
 $g_{ij}$ of $g$ in this chart satisfy

\begin{equation}\frac{1}{(1+\epsilon)}\delta_{ij}\leq
g_{ij}\leq(1+\epsilon)\delta_{ij}\label{eq:02}
\end{equation}
in the sense of bilinear form (see \cite{Aubi}). We claim that by
choosing $\epsilon>0$ small enough we get by \eqref{eq:01} that
there exist $\delta_{0}>0$ and $C'<C_{opt}(\mathbb{R}^{n})$ such
that $\forall f\in \mathcal{C}_{0}^{\infty}(B_{\delta_{0}}(0))$,

\begin{equation}\label{eq:103}\left(\int_{B_{\delta_{0}}(0)}|x|^{\gamma r}|f|^{r}dx\right)^{\frac{1}{r}}\leq C'\left(\int_{B_{\delta_{0}}(0)}|x|^{\alpha p}|\nabla
f|^{p}dx\right)^{\frac{a}{p}}\left(\int_{B_{\delta_{0}}(0)}|x|^{\beta
q}|f|^{q}dx\right)^{\frac{1-a}{q}}.
\end{equation}

Indeed, if $f\in \mathcal{C}_{0}^{\infty}(B_{\delta_{0}}(0))$, then
$u:=f\circ exp_{p}^{-1}\in\mathcal{C}_{0}^{\infty}(\Omega)$.
Substituting $u$ into \eqref{eq:01} and using the metric estimates
\eqref{eq:02}, we obtain

\begin{eqnarray}
\left(\int_{B_{\delta_{0}}(0)}|x|^{\gamma
r}|f|^{r}dx\right)^{\frac{1}{r}}\leq
C'\left(\int_{B_{\delta_{0}}(0)}|x|^{\alpha p}|\nabla
f|^{p}dx\right)^{\frac{a}{p}}\left(\int_{B_{\delta_{0}}(0)}|x|^{\beta
q}|f|^{q}dx\right)^{\frac{1-a}{q}}\nonumber
\end{eqnarray}
where
$C'=(1+\epsilon)^{\frac{n}{2r}}(1+\epsilon)^{\frac{an}{2p}+\frac{n(1-a)}{2q}+\frac{a}{2}}C$.
Since $C<C_{opt}(\mathbb{R}^{n})$ we know that if $\epsilon$ is
small enough then $C'<C_{opt}(\mathbb{R}^{n})$. This proves our
claim.

Let $u\in\mathcal{C}_{0}^{\infty}(\mathbb{R}^{n})$. Set
$u_{\lambda}(x)=u(\lambda x)$, $\lambda>0$. For $\lambda$ large
enough $u_{\lambda}(x)\in
\mathcal{C}_{0}^{\infty}(B_{\delta_{0}}(0))$. Substituting
$u_{\lambda}$ into \eqref{eq:103}, we get

\begin{equation}\label{eq:03}\left(\int_{\mathbb{R}^{n}}|x|^{\gamma r}|u_{\lambda}|^{r}dx\right)^{\frac{1}{r}}\leq C'\left(\int_{\mathbb{R}^{n}}|x|^{\alpha p}|\nabla
u_{\lambda}|^{p}dx\right)^{\frac{a}{p}}\left(\int_{\mathbb{R}^{n}}|x|^{\beta
q}|u_{\lambda}|^{q}dx\right)^{\frac{1-a}{q}}.
\end{equation}
Using a change of variables, we have

\begin{equation}\left(\int_{\mathbb{R}^{n}}|x|^{\gamma r}|u_{\lambda}(x)|^{r}dx\right)^{\frac{1}{r}}=\lambda^{-\frac{n}{r}}\lambda^{-\gamma}\left(\int_{\mathbb{R}^{n}}|y|^{\gamma r}|u(y)|^{r}dy\right)^{\frac{1}{r}},\nonumber
\end{equation}

\begin{equation}\left(\int_{\mathbb{R}^{n}}|x|^{\alpha p}|\nabla
u_{\lambda}(x)|^{p}dx\right)^{\frac{a}{p}}=\lambda^{-\frac{na}{p}}\lambda^{-\alpha
a}\lambda^{a}\left(\int_{\mathbb{R}^{n}}|y|^{\alpha p}|\nabla
u(y)|^{p}dy\right)^{\frac{a}{p}}\nonumber
\end{equation}
and

\begin{equation}\left(\int_{\mathbb{R}^{n}}|x|^{\beta q}|u_{\lambda}(x)|^{q}dx\right)^{\frac{1-a}{q}}=\lambda^{-\frac{n(1-a)}{q}}\lambda^{-\beta(1-a)}\left(\int_{\mathbb{R}^{n}}|y|^{\beta q}|u(y)|^{q}dy\right)^{\frac{1-a}{q}}.\nonumber
\end{equation}
Combining the above equations with \eqref{eq:03}, we get

\begin{equation}\left(\int_{\mathbb{R}^{n}}|y|^{\gamma r}|u(y)|^{r}dy\right)^{\frac{1}{r}}\leq
C'\lambda^{-\alpha
a-\frac{na}{p}+a-\frac{n(1-a)}{q}-\beta(1-a)+\frac{n}{r}+\gamma}\left(\int_{\mathbb{R}^{n}}|y|^{\alpha
p}|\nabla
u(y)|^{p}dy\right)^{\frac{a}{p}}\left(\int_{\mathbb{R}^{n}}|y|^{\beta
q}|u(y)|^{q}dy\right)^{\frac{1-a}{q}}.\nonumber
\end{equation}
It follows from the conditions of Theorem \ref{eq:0} that

\begin{eqnarray}
-\alpha
a-\frac{na}{p}+a-\frac{n(1-a)}{q}-\beta(1-a)+\frac{n}{r}+\gamma=0.\nonumber
\end{eqnarray}
Hence, we have

\begin{equation}\left(\int_{\mathbb{R}^{n}}|y|^{\gamma r}|u(y)|^{r}dy\right)^{\frac{1}{r}}\leq
C'\left(\int_{\mathbb{R}^{n}}|y|^{\alpha p}|\nabla
u(y)|^{p}dy\right)^{\frac{a}{p}}\left(\int_{\mathbb{R}^{n}}|y|^{\beta
q}|u(y)|^{q}dy\right)^{\frac{1-a}{q}}.\nonumber
\end{equation}
This expression contradicts the fact that $C_{opt}(\mathbb{R}^{n})$
is the best constant for this inequality on $\mathbb{R}^{n}$.
\end{proof}

\end{section}

\begin{section}{Proof of Theorems \ref{eq:58}, \ref{114} and \ref{117}}
In this section, we will briefly mention some basic definitions and
notions in Finsler geometry. There are many good references in the
subject, we refer readers to \cite{Bao} and \cite{Shen}.

\begin{subsection}{Finsler Geometry}

\begin{definition}(Finslerian Structure)A Finslerian structure is a pair
$(M^{n},F)$ consisting of a connected $C^{\infty}$ manifold and a
continuous function
\begin{equation}F:TM\rightarrow [0,\infty)\nonumber
\end{equation}
satisfying the following properties
\begin{itemize}
  \item $F\in \mathcal{C}^{\infty}(TM-\{0\});$
  \item $F(x,ty)=tF(x,y),\hspace{0,2cm}\forall
  t\geq0\hspace{0,2cm} \text{and} \hspace{0,2cm}(x,y)\in TM;$
  \item The $n\times n$ matrix
  \begin{equation}\label{eq:55}(g_{ij}):=\Bigg{(}\Big{[}\frac{1}{2}F^{2}\Big{]}_{\partial y^{i}\partial y^{j}}\Bigg{)}, \hspace{0,2cm}y=\sum_{i=1}^{n}y^{i}\frac{\partial}{\partial x^{i}}
  \end{equation}
  is positive definite for all $(x,y)\in TM-\{0\}.$
\end{itemize}
\end{definition}
A Finsler manifold $(M,F)$ is called a \textit{locally Minkowski
space} if there exist certain privileged local coordinate system
$(x^{i})$ on $M$, such that in each coordinated neighborhood we have
that $F(x,y)$ depends only on $y$ and not on $x$. On the other hand,
a \textit{Minkowski space} consist of a finite dimensional vector
space $V$ and a Minkowski norm which induces a Finsler metric on $V$
by translation.

We consider on the pull-back bundle $\pi^{\ast} TM$ the Chern
connection [see Bao et al. \cite{Bao}, Theorem 2.4.1]. The
coefficients of the Chern connection are given by

\begin{equation}\label{eq3}\Gamma_{jk}^{i}(x,y)=\frac{1}{2}g^{il}\left({\frac{\partial
g_{lj}}{\partial x_{k}}}-{\frac{\partial g_{jk}}{\partial
x_{l}}}+{\frac{\partial g_{kl}}{\partial x_{j}}}-{\frac{\partial
g_{ij}}{\partial y_{r}}}G_{k}^{r}+{\frac{\partial g_{jk}}{\partial
y_{r}}}G_{l}^{r}-{\frac{\partial g_{kl}}{\partial
y_{r}}}G_{j}^{r}\right)
\end{equation}
where $G_{j}^{i}=\frac{\partial G^{i}}{\partial y^{j}}$ and

$$G^{i}(x,y)=\frac{1}{4}g^{ik}\left({2\frac{\partial g_{jk}}{\partial
x_{l}}}-{\frac{\partial g_{jl}}{\partial x_{k}}}\right)y_{i}y_{j}.$$

With this connection we consider the following space
\begin{definition}(Berwaldian Structure) A Finsler manifold is a Berwald space if the
coefficients of $\Gamma_{jk}^{i}(x,y)$ given by expression
\eqref{eq3} in natural coordinates are independent of $y$.
\end{definition}

A geodesic between two points $x,y\in M$ is a smooth curve
$\tau:[0,1]\rightarrow\mathbb{R}$ minimizing the following
functional

\begin{equation}\sigma\mapsto L_{F}(\tau)=\int_{0}^{l}F(\tau,\dot{\tau})dt\nonumber
\end{equation}
and the distance function is given by
$d_{F}(x_{1},x_{2}):=inf_{\tau}L_{F}(\tau)$, where $\tau$ varies
over all smooth curves connecting $x_{1}$ to $x_{2}$. A Finsler
manifold $(M,F)$ is said to be complete if any geodesic
$\tau:[0,l]\rightarrow M$ can be extended to a geodesic
$\tau:\mathbb{R}\rightarrow M$.

Let $\tau:[0,l]\rightarrow M$ be a geodesic with velocity field
$\dot{\tau}$. A vector field $J$ along $\tau$ is said to be a
\textit{Jacobi field} if it satisfies the equation

\begin{equation}\label{eq:57}D_{\dot{\tau}}^{\dot{\tau}}D_{\dot{\tau}}^{\dot{\tau}}J+R^{\dot{\tau}}(J,\dot{\tau})\dot{\tau}=0
\end{equation}
where $D^{\dot{\tau}}$ is the covariant derivative with reference
vector $\dot{\tau}$, and $R^{\dot{\tau}}$ is the curvature tensor
(see \cite{Bao} for details).

For a flag $\mathcal{P}:=span\{v,w\}\subset T_{x}M$, with flag pole
$v$, the \textit{flag curvature} is defined by

\begin{equation}K(\mathcal{P},v):=\frac{\langle R^{v}(w,v)v,w\rangle_{v}}{F(v)^{2}\langle w,w\rangle_{v}-\langle
v,w\rangle_{v}^{2}}\nonumber
\end{equation}
where $\langle,\rangle_{v}$ denotes the inner product induced by
\eqref{eq:55}. In the Riemannian case the flag curvature reduces to
the sectional curvature  which depends only on $\mathcal{P}$.

Consider $v\in T_{x}M$ with $F(x,v)=1$ and let $\{e_{i}\}_{i=1}^{n}$
with $e_{n}=v$ be an orthonormal basis of
$(T_{x}M,\langle,\rangle_{v})$. Put
$\mathcal{P}_{i}=span\{e_{i},v\}$ for $1\leq i\leq n-1$. Then the
Ricci curvature of $v$ is defined by
\begin{equation}Ric(v):=\sum_{i=1}^{n-1}K(\mathcal{P}_{i},v).\nonumber
\end{equation}
For $c\geq0$, we also set $Ric(cv):=c^{2}Ric(v)$.

Motivated by the work of Lott-Villani \cite{Lott} and Sturm
\cite{Sturm} on metric measure space, Ohta in \cite{Oh} introduce
the notion of weighted Ricci curvature on Finsler manifolds as
follow; consider $\textsf{m}$ be a positive measure on $(M,F)$,
given a unit vector $v\in T_{x}M$ extend it to a $C^{\infty}$ vector
field $V$ on a neighborhood $U_{x}$ of $x$ such that every integral
curve is a geodesic, and decompose $\textsf{m}$ as
$\textsf{m}=e^{-\psi}vol_{V}$ on $U_{x}$, where $vol_{V}$ denotes
the volume form of the Riemannian structure $g_{V}$. In what follows

\begin{definition}(\textrm{Weighted Ricci Curvature})For $N\in[n,\infty]$ and a unit vector $v\in
T_{p}M$ the $N$-Ricci curvature $Ric_{N}$ is defined by

\begin{enumerate}
  \item  $Ric_{n}(v) := \begin{cases}
                        Ric(v)+(\psi\circ\sigma)''(0)\hspace{0,2cm} \text{if}\hspace{0,2cm} (\psi\circ\sigma)'(0)=0\\
                        -\infty \hspace{0,2cm}\text{otherwise}\end{cases}$
  \item $Ric_{N}(v):=Ric(v)+(\psi\circ\sigma)''(0)-\frac{(\psi\circ\sigma)'(0)^{2}}{N-n}\hspace{0,2cm}\text{for} \hspace{0,2cm}N\in (n,\infty)$
  \item $Ric_{\infty}(v):=Ric(v)+(\psi\circ\sigma)''(0)$
\end{enumerate}
For $c\geq0$, we also define $Ric_{N}(cv):=c^{2}Ric_{N}(v)$.
\end{definition}

Inspired by the $Ric_{N}$ concept, Ohta in \cite{Oh} proved the
following Bishop-Gromov-type volume comparison theorem.

\begin{theorem}\label{eq:56}{\textrm{(\cite{Oh}, Theorem 7.3)}} Let $(M,F,\textsf{m})$
be a complete $n$-dimensional Finsler manifold with non-negative
$N$-Ricci curvature. Then we have
\begin{equation}\frac{\textsf{m}(B_{R}(x))}{\textsf{m}(B_{\rho}(x))}\leq\left(\frac{R}{\rho}\right)^{N},\hspace{0,2cm}\forall
x\in M, \hspace{0,2cm} e \hspace{0,2cm}0<\rho<R.
\end{equation}
Moreover, if equality holds with $N=n$ for all $x\in M$ and $0<r<R$,
then any Jacobi field $J$ along a geodesic $\tau$ has the form
$J(t)=tP(t)$, where $P$ is a parallel vector field along $\tau$.
\end{theorem}

\end{subsection}

\begin{myproof}[\textbf{Proof of Theorem \ref{eq:58}}]Since $(M,F)$ is complete,
by the Hopf-Rinow theorem it yields that $(M,d_{F},\textsf{m})$ is a
proper metric measure space. On account of Theorem \ref{eq:56} we
have that the condition \eqref{eq:32} in Theorem \ref{eq:45} holds
with $C_{0}=1$. Note that the choice of constant $1$ on the right
side of \eqref{eq:33} was done for simplicity. In fact, by
\eqref{eq:32} we have that $\Lambda_{x_{0}}=\liminf_{\rho\rightarrow
0}\frac{\textsf{m}(B_{\rho}(x_{0}))}{\textsf{m}_{E}(\mathbb{B}_{\rho}(0))}$
is positive. Then we can normalize the measure $\textsf{m}$ in order
to satisfy \eqref{eq:33}.

The condition \eqref{eq:32}, implies that

\begin{equation}\frac{\textsf{m}(B_{R}(x))}{w_{n}R^{n}}\leq
\frac{\textsf{m}(B_{\rho}(x))}{w_{n}\rho^{n}}=\frac{\textsf{m}(B_{\rho}(x))}{\textsf{m}_{E}(\mathbb{B}_{\rho}(0))},\hspace{0,2cm}
0<\rho<R.\nonumber
\end{equation}
Taking $\rho\rightarrow0$, we have from 

$$\liminf_{\rho\rightarrow0}\frac{\textsf{m}(B_{\rho}(x))}{w_{n}\rho^{n}}=1$$
that $\textsf{m}(B_{R}(x))\leq w_{n}R^{n}$, $\forall x\in M$ and
$R>0$.

Since the sharp Caffarelli-Kohn-Nirenberg inequality holds, we have
from Theorem \ref{eq:45}, the reverse inequality
$\textsf{m}(B_{R}(x))\geq w_{n}R^{n}$, $\forall x\in M$ and $R>0$.
Thus, $\textsf{m}(B_{R}(x))=\textsf{m}_{E}(\mathbb{B}_{R}(0))$,
$\forall x\in M$ and $R>0$. By Theorem \ref{eq:56}, it results that
every Jacobi field $J$ along any geodesic $\tau$ has the form
$J(t)=tP(t)$, where $P$ is a parallel vector field along $\tau$.
Then it follows from the Jacobi equation \eqref{eq:57} that
$R^{\dot{\tau}}(J,\dot{\tau})\equiv0$, so that
$K(\mathcal{P},\dot{\tau})\equiv0$ with
$\mathcal{P}=span\{P,\dot{\tau}\}$. Due to the arbitrariness of
$\tau$ and $J$, it turns out that the flag curvature of $(M,F)$ is
identically zero.

\end{myproof}

\begin{myproof}[\textbf{Proof of Theorem \ref{114}}]Since $(M,F)$ is a
Berwald space, the non-negativity of the Ricci curvature on $(M,F)$
coincides with the non-negativity of the $n$-Ricci curvature on
$(M,d_{F},\textsf{m}_{BH})$, and the Busemann-Hausdorff measure
$\textsf{m}_{BH}$ satisfies the following $n$-density assumption
\begin{equation}\lim_{\rho\rightarrow0}\frac{\textsf{m}_{BH}(B_{\rho}(x))}{w_{n}\rho^{n}}=1,\nonumber
\end{equation}
see Shen \cite{Shen1} Lemma 5.2 and \cite{Oh} Theorem 1.2. Then,
applying Theorem \ref{eq:58}, we get that the flag curvature of
$(M,F)$ is identically zero. On the other hand, every Berwald space
with zero flag curvature is necessarily a locally Minkowski space,
see \cite{Bao} section 10.5. Due to the volume identity
$\textsf{m}_{BH}(B_{\rho}(x))=w_{n}\rho^{n}$, $\forall x\in M$ and
$\rho>0$, we have that $(M,F)$ must be isometric to a Minkowski
space.

\end{myproof}

\begin{myproof}[\textbf{Proof of Theorem \ref{117}}]Since $(M,F)$ is complete,
by the Hopf-Rinow theorem it yields that $(M,d_{F},\textsf{m})$ is a
proper metric measure space. The non-negativity of the flag
curvature implies that the Ricci curvature is non-negative, then in
the same way as in the proof of Theorem \ref{114}, we have that the
$n$-Ricci curvature is non-negative, and by Theorem \ref{eq:56} the
condition \eqref{eq:32} in Theorem \ref{eq:45} holds with $C_{0}=1$.
Now, since the Busemann-Hausdorff measure $\textsf{m}_{BH}$
satisfies
\begin{equation}\lim_{\rho\rightarrow0}\frac{\textsf{m}_{BH}(B_{\rho}(x))}{w_{n}\rho^{n}}=1,\nonumber
\end{equation}
we have by Theorem \ref{eq:45} that

\begin{equation}
0<\left(\frac{C_{opt}(\mathbb{R}^{n})}{C}\right)^{\frac{n}{a}}\leq
\frac{\textsf{m}_{BH}(B_{\rho}(x_{0}))}{w_{n}\rho^{n}}\leq
1,\hspace{0,2cm}\forall\rho>0.\nonumber
\end{equation}
This inequality implies that $(M,F)$ has large volume growth as
defined by Lakzian (see \cite{Lak} Definition 3.6), then by
Remark \ref{eq4}, $M$ is diffeomorphic to Euclidean space
$\mathbb{R}^{n}$.

\end{myproof}

\end{section}


\begin{section}{Proof of Theorems \ref{111} and \ref{112}}

\begin{subsection}{Geometry on Alexandrov Spaces}

\hspace{0,5cm}In this section, by completeness, we define the notion
of Alexandrov space. First, let us remember that a length space
$(X,d_{X})$ is a metric space where the distance function $d_{X}$
between two points is given by the infimum of the lengths of all the
curves connecting these two points. A triangle in $(X,d_{X})$
consists of three points $x,y,z$ and three minimal geodesics
$\overline{xy}, \overline{xz}, \overline{yz}$. Fix a real number
$\kappa\in\mathbb{R}$, a comparison triangle
$\tilde{x}\tilde{y}\tilde{z}$ is a triangle on the surface of
constant curvature $\kappa$, with the same side lengths. We denote
this comparison angles by $\widetilde{\angle}_{\kappa}xyz$,
$\widetilde{\angle}_{\kappa}yzx$ and
$\widetilde{\angle}_{\kappa}zxy$. A comparison triangle exists and
is unique whenever $\kappa\leq0$ or $\kappa>0$ and
$|xy|+|xz|+|zy|<\frac{2\pi}{\sqrt{\kappa}}$.

\begin{definition}\label{eq:63}An Length Space $X$ is called an
Alexandrov space of curvature $\geq \kappa$ if any $x_{0}\in X$ has
a neighborhood $U_{x_{0}}$, such that for any $x,y,z,w\in U_{x_{0}}$

\begin{equation}\widetilde{\angle}_{\kappa}yxz+\widetilde{\angle}_{\kappa}zxw+\widetilde{\angle}_{\kappa}wxy\leq2\pi.\nonumber
\end{equation}
\end{definition}

For locally compact spaces this is equivalent to the more familiar
Alexandrov-Toponogov distance comparison.

\begin{definition}\label{eq:62}(Toponogov-Alexandrov)A locally compact Length Space $X$
is called an Alexandrov space of curvature $\geq \kappa$ if any
$x_{0}\in X$ has a neighborhood $U_{x_{0}}$, such that for any
triangle $xyz$ in $U_{x_{0}}$ and any $y_{1}\in\overline{xy}$,
$z_{1}\in\overline{xz}$, we have
$|y_{1}z_{1}|\geq|\widetilde{y_{1}}\widetilde{z_{1}}|$, where
$\widetilde{y_{1}}$ and $\widetilde{z_{1}}$ are the corresponding
points on the sides $\widetilde{x}\widetilde{y}$ and
$\widetilde{x}\widetilde{z}$ of the comparison triangle
$\widetilde{x}\widetilde{y}\widetilde{z}$.
\end{definition}

\begin{observation}If $X$ is complete, the local condition in the above definitions implies a global condition.
\end{observation}

Similar to the Bishop-Gromov comparison theorem in Riemannian
manifolds, there is an extension for Alexandrov spaces. The next
result can be found in \cite{Bura}(see theorem $10.6.6$).

\begin{theorem}\label{eq:59}{(Bishop-Gromov Inequality)} Let $X$ be an Locally compact $n$-dimensional Alexandrov space of curvature $\geq \kappa$. Then for any $x\in
X$ the function
\begin{equation}\rho\rightarrow\frac{\mathcal{H}^{n}(B_{\rho}(x))}{V_{\rho}^{\kappa}}\nonumber
\end{equation}
is not increasing, where $\mathcal{H}^{n}(B_{\rho}(x))$ is the
$n$-dimensional Hausdorff measure of the ball of radius $\rho$ and
center $x$ in space form $M_{\kappa}^{n}$. That is, if $R\geq
\rho>0$, then
\begin{equation}\frac{\mathcal{H}^{n}(B_{R}(x))}{V_{R}^{\kappa}}\leq\frac{\mathcal{H}^{n}(B_{\rho}(x))}{V_{r}^{\kappa}}. \nonumber
\end{equation}
\end{theorem}

Next, Kuwae et al in \cite{Kuw} define the concept of
\textit{Infinitesimal Bishop-Gromov inequality} for Alexandrov
spaces as follows. For a real number $\kappa$, consider

$$\begin{displaystyle}
  s_{\kappa}(\rho) =
  \begin{cases}
    \frac{\sin(\sqrt{\kappa}\rho)}{\sqrt{\kappa}}, & \text{if } \kappa>0 \\
    \rho, & \text{if } \kappa=0 \\
    \frac{\sinh(\sqrt{|\kappa|}\rho)}{\sqrt{|\kappa|}}, & \text{if }
    \kappa<0
  \end{cases}
\end{displaystyle}$$
observe that the function $s_{\kappa}$ is a solution of Jacobi
equation $s_{\kappa}''(\rho)+\kappa s_{\kappa}'(\rho)=0$ with
initial conditions $s_{\kappa}(0)=0$ and $s_{\kappa}'(0)=1$. Let
$d_{x_{0}}(x):=d(x_{0},x)$, where $x_{0},x\in X$ and $d$ is the
distance function. For $x_{0}\in X$ and $0<t\leq1$, we define the
set $W_{x_{0},t}\subset X$ and the map
$\Phi_{x_{0},t}:W_{x_{0},t}\rightarrow X$ as following: First, put
$\Psi_{x_{0},t}(x_{0})=x_{0}\in W_{x_{0},t}$. A point $x(\neq
x_{0})$ belongs to $W_{x_{0},t}$ if, and only if, there exists $y\in
X$ such that $x\in x_{0}y$ and $d_{x_{0}}(x):d_{x_{0}}(y)=t:1$,
where $x_{0}y$ is a minimal geodesic connecting $x_{0}$ to $y$.
Since a geodesic does not branch on an Alexandrov space, for a given
point $x\in W_{x_{0},t}$ such a point $y$ is unique and we set
$\Psi_{x_{0},t}(x)=y$. Now we are in a position to define the notion
of infinitesimal Bishop-Gromov inequality.

\begin{definition}Given a real numbers $n\geq1$ and $\kappa$, we say that the $n$-dimensional Hausdorff measure $\mathcal{H}^{n}$ satisfies
the Bishop-Gromov infinitesimal inequality $BG(\kappa,n)$ if for any
$x_{0}\in X$ and $t\in(0,1]$ we have
\begin{equation}d(\Psi_{x_{0},t\ast}\mathcal{H}^{n})(x)\leq\frac{ts_{\kappa}(td_{x_{0}}(x))^{n-1}}{s_{\kappa}(d_{x_{0}}(x))^{n-1}}d\mathcal{H}^{n}(x)\nonumber
\end{equation}
for all $x\in X$ such that $d_{x_{0}}(x)<\frac{\pi}{\sqrt{\kappa}}$
if $\kappa>0$, where $\Psi_{x_{0},t\ast}\mathcal{H}^{n}$ is the
push-forward of $\mathcal{H}^{n}$ by $\Psi_{x_{0},t}$.
\end{definition}

$BG(\kappa,n)$ is sometimes called the measure contraction
property(see \cite{Oh00, Kuw, Sturm, Sturm1} )and is weaker than the
curvature dimension(or lower $n$-Ricci curvature)condition
$CD((n-1)\kappa,n)$ introduced by Sturm \cite{Sturm, Sturm1}.

In \cite{Kuw}, the authors show that

\begin{theorem}\label{eq:60}Let $X$ be an $n$-dimensional Alexandrov space of curvature $\geq \kappa$. Then, the $n$-dimensional
Hausdorff measure $\mathcal{H}^{n}$ on $X$ satisfies the
infinitesimal Bishop-Gromov condition $BG(\kappa,n)$.
\end{theorem}

Let us denote by $Alex^{n}[\kappa]$ the class of $n$-dimensional
Alexandrov space of curvature $\geq \kappa$. In \cite{Munn} see
Theorem $3.2$, the authors prove the following

\begin{theorem}\label{eq:61}For an integer $n\geq2$, let $(X,d)\in
Alex^{n}[-\kappa^{2}]$, $\kappa\in\mathbb{R}$ be a complete
non-compact Alexandrov space whose Hausdorff measure
$\mathcal{H}^{n}$ satisfies the $BG(0,n)$ condition. There exists an
$\epsilon(n,k)=\epsilon>0$ such that, if $x\in X$
\begin{equation}\mathcal{H}^{n}(B_{\rho}(x))\geq(1-\epsilon)w_{n}\rho^{n},\hspace{0,2cm}
\forall \rho>0.\nonumber
\end{equation}
Then (X,d) has finite topological type.
\end{theorem}

\begin{myproof}[\textbf{Proof of Theorem \ref{111}}]Since $X$ is Locally compact and complete, we have that closed and bounded subset of $X$ are compact,
Thus $X$ is a proper space. Now, since $X$ has curvature $\geq0$, we
can apply Theorem \ref{eq:59} to get
\begin{equation}\frac{\lambda\mathcal{H}^{n}(B_{R}(x))}{\lambda\mathcal{H}^{n}(B_{\rho}(x))}=\frac{\mathcal{H}^{n}(B_{R}(x))}{\mathcal{H}^{n}(B_{\rho}(x))}\leq\frac{R^{n}}{\rho^{n}},\hspace{0,2cm} x\in X,\hspace{0,2cm}0<\rho<R. \nonumber
\end{equation}
Thus, the condition \eqref{eq:32} of Theorem \ref{eq:45} is
satisfied with $C_{0}=1$.

By Lemma $3.2$ of \cite{Shi}, we get
\begin{equation}\lim_{\rho\rightarrow0}\frac{\mathcal{H}^{n}(B_{\rho}(x_{0}))}{\rho^{n}}=\mathcal{H}^{n}(B_{1}(o_{x_{0}})).\nonumber
\end{equation}
Thus,
\begin{equation}\liminf_{\rho\rightarrow0}\frac{\lambda\mathcal{H}^{n}(B_{\rho}(x_{0}))}{w_{n}\rho^{n}}=\lambda\frac{\mathcal{H}^{n}(B_{1}(0_{x_{0}}))}{w_{n}}=1.\nonumber
\end{equation}
Applying Theorem \ref{eq:45}, we obtain

\begin{equation}
w_{n}\rho^{n}\leq \lambda\mathcal{H}^{n}(B_{\rho}(x))\leq
w_{n}\rho^{n},\hspace{0,2cm}\forall \rho>0.\nonumber
\end{equation}
Hence $\lambda\mathcal{H}^{n}(B_{\rho}(x))=w_{n}\rho^{n}$, $\forall
\rho>0$. This implies that $X$ is isometric to Euclidean space.
\end{myproof}

\begin{myproof}[\textbf{Proof of Theorem \ref{112}}]Consider $\epsilon>0$ given by Theorem
\ref{eq:61}. Since the function $\phi:[0,1]\rightarrow\mathbb{R}$
defined by
\begin{equation}
\phi(x):=\left(\frac{C_{opt}(\mathbb{R}^{n})}{C_{opt}(\mathbb{R}^{n})+x}\right)^{\frac{n}{a}}
\nonumber
\end{equation}
converge to $1$ when $x\rightarrow0$, we have that there exist a
$\delta>0$ such that
\begin{equation}
0<x\leq\delta\Longrightarrow1-\epsilon\leq\left(\frac{C_{opt}(\mathbb{R}^{n})}{C_{opt}(\mathbb{R}^{n})+x}\right)^{\frac{n}{a}}.
\nonumber
\end{equation}
Then, applying Theorems \ref{eq:45}, \ref{eq:60} and
\ref{eq:61} we get the result.

\end{myproof}

\end{subsection}

\end{section}

\vskip0.8cm
\noindent
{Willian Isao Tokura } (e-mail: williamisaotokura@hotmail.com)\\[2pt]
Instituto de Matem\'atica e Estat\'istica\\
Universidade Federal de Goi\'as\\
74001-900-Goi\^ania-GO\\
Brazil\\

\noindent{Levi Adriano } (e-mail: levi@ufg.br)\\[2pt]
Instituto de Matem\'atica e Estat\'istica\\
Universidade Federal de Goi\'as\\
74001-900-Goi\^ania-GO\\
Brazil\\

\noindent{Changyu Xia} (e-mail: xia@mat.unb.br)\\[2pt]
Departamento de Matem\'atica\\
Universidade de Brasilia\\
70910-900-Brasilia-DF\\
Brazil


\begin{thebibliography}{X}
\bibitem{Adr}ADRIANO, Levi; XIA, Changyu. Sobolev type inequalities on Riemannian manifolds. \emph{Journal of Mathematical Analysis and Applications}. \textbf{371} (2010), 372--383.

\bibitem{Adr2}ADRIANO, Levi; XIA, Changyu. Hardy type inequalities on complete Riemannian manifolds. \emph{Monatshefte für Mathematik}. \textbf{163} (2011), 115--129.

\bibitem{Aubi} AUBIN, Thierry. Some nonlinear problems in Riemannian geometry. \emph{Springer Science, Business and Media}. 2013.


manifolds. \emph{Journal of Mathematical Analysis and Applications}.
\textbf{349} (2009), 493--502.

\bibitem{Bao}BAO, D.; CHERN, S. S.; SHEN, Z. An Introduction to Riemannian Finsler Geom. \emph{Graduate Texts in Math.} \textbf{200}, 2000.

\bibitem{Bura}BURAGO, Dmitri; BURAGO, Yuri; IVANOV, Sergei. A course in metric geometry. \emph{Providence: American Mathematical Society}. 2001.

\bibitem{Caff} CAFFARELLI, Luis; KOHN, Robert; NIRENBERG, Louis. First order interpolation inequalities with weights. \emph{Compositio Mathematica}. \textbf{53} (1984), 259--275.

\bibitem{Chavel}CHAVEL, Isaac. Riemannian geometry: a modern introduction. \emph{Cambridge university press}. 2006.

\bibitem{Chee}CHEEGER, Jeff et al. On the structure of spaces with Ricci curvature bounded below. I. \emph{Journal of Differential Geometry}. \textbf{46} (1997), 406--480.

\bibitem{doCarmo}DO CARMO, Manfredo Perdigão; XIA, Changyu. Complete manifolds with non-negative Ricci curvature and the Caffarelli-Kohn-Nirenberg inequalities. \emph{Compositio Mathematica}. \textbf{140} (2004), 818--826.



\bibitem{Heb}HEBEY, Emmanuel. Nonlinear analysis on manifolds: Sobolev spaces and inequalities. \emph{American Mathematical Soc.} 2000.

\bibitem{Kell}KELL, Martin. A note on non-negatively curved Berwald spaces. arXiv preprint arXiv:1502.03764, 2015.

\bibitem{Kris}KRISTÁLY, Alexandru; OHTA, Shin-ichi. Caffarelli-Kohn-Nirenberg inequality on metric measure spaces with applications. \emph{Mathematische Annalen}. \textbf{357} (2013), 711--726.

\bibitem{Kris1}KRISTÁLY, Alexandru. Metric measure spaces supporting Gagliardo-Nirenberg inequalities: volume non-collapsing and rigidities. \emph{Calculus of Variations and Partial Differential Equations}. \textbf{55} paper n.o 112 (2016), 27pp.

\bibitem{Kuw}KUWAE, Kazuhiro; SHIOYA, Takashi. Infinitesimal Bishop-Gromov condition for Alexandrov spaces. \emph{Adv. Stud. Pure Math.} \textbf{57} (2010), 293--302.

\bibitem{lam}LAM, Nguyen; LU, Guozhen. Sharp constants and optimizers for a class of the Caffarelli-Kohn-Nirenberg inequalities.\emph{ Advanced Nonlinear Studies, 0(0), pp. -. Retrieved 14 Jun. 2017, from doi:10.1515/ans-2017-0012}, 2017.

\bibitem{Led}LEDOUX, M. On manifolds with non-negative Ricci curvature and Sobolev inequalities. \emph{Comm. Anal. Geom.} \textbf{7} (1999), 347--353.

\bibitem{Lott}LOTT, John; VILLANI, Cédric. Ricci curvature for metric-measure spaces via optimal transport. Annals of Mathematics, p. 903-991, 2009.

\bibitem{Munn}MUNN, Michael. Alexandrov spaces with large volume growth. \emph{Journal of Mathematical Analysis and Applications}. \textbf{419} (2014), 525--540.



\bibitem{Oh}OHTA, Shin-ichi. Finsler interpolation inequalities. \emph{Calculus of Variations and Partial Differential Equations} \textbf{36} (2009), 211--249.

\bibitem{Oh00}OHTA, Shin-ichi. On the measure contraction property of metric measure spaces. Commentarii Mathematici Helvetici, v. 82, n. 4, p. 805-828, 2007.



\bibitem{Rig} PIGOLA, Stefano; RIGOLI, Marco; SETTI, Alberto G. Vanishing and finiteness results in geometric analysis: a generalization of the Bochner technique. \emph{Springer Science and Business Media}. 2008.

\bibitem{Lak}LAKZIAN, Sajjad. On closed geodesics in non-negatively curved Finsler structures with large volume growth. arXiv preprint arXiv:1408.0214, 2014.

\bibitem{Schoen} SCHOEN, Richard; YAU, S. Lectures on differential geometry vol. 1 of Conference Proceedings and Lecture Notes in Geometry and Topology. International Press, 1994.

\bibitem{Shen1}SHEN, Zhongmin. Volume comparison and its applications in Riemann-Finsler geometry. \emph{Advances in Mathematics}. \textbf{128} (1997), 306--328.

\bibitem{Shen}SHEN, Zhongmin. Lectures on Finsler geometry. World Scientific, 2001.

\bibitem{Shi}SHIOYA, Takashi. Mass of rays in Alexandrov spaces of nonnegative curvature. \emph{Commentarii Mathematici Helvetici}. \textbf{69} (1994), 208--228.

\bibitem{Sturm}STURM, Karl-Theodor. On the geometry of metric measure spaces I. Acta mathematica, v. 196, n. 1, p. 65-131, 2006.

\bibitem{Sturm1}STURM, Karl-Theodor. On the geometry of metric measure spaces II. Acta mathematica, v. 196, n. 1, p. 133-177, 2006.

\bibitem{Xia4}XIA, Changyu. The Caffarelli-Kohn-Nirenberg inequalities on complete manifolds. \emph{Mathematical Research Letters}. \textbf{14} (2007), 875--885.

\bibitem{Xia1}XIA, Changyu et al. Complete manifolds with non-negative Ricci curvature and almost best Sobolev constant. \emph{Illinois Journal of Mathematics}. \textbf{45} (2001), 1253--1259.

\bibitem{Xia}XIA, Changyu. The Gagliardo-Nirenberg inequalities and manifolds of non-negative Ricci curvature. \emph{Journal of Functional Analysis} \textbf{224} (2005), 230--241.

\bibitem{Zhu}ZHU, Shun-Hui. A volume comparison theorem for manifolds with asymptotically nonnegative curvature and its applications. \emph{American Journal of Mathematics}. \textbf{116} (1994), 669--682.

\end{thebibliography}
\end{document}